\documentclass[review,11 pt]{elsarticle}

\usepackage{hyperref}
\usepackage{amsmath,amssymb}
\newtheorem{thm}{Theorem}[section]
\newtheorem{defn}[thm]{Definition}
\newtheorem{prop}[thm]{Proposition}
\newtheorem{cor}[thm]{Corollary}
\newtheorem{lem}[thm]{Lemma}
\newtheorem{remark}[thm]{Remark}

\newtheorem{ex}[thm]{Example}

\newproof{proof}{Proof}

\begin{document}

\begin{frontmatter}

\title{Characterizations of woven frames}

\author[NITM]{A. Bhandari}
\ead{animesh@nitm.ac.in}

\author[NITM]{S. Mukherjee\corref{cor}\fnref{fn}}
\ead{saikat.mukherjee@nitm.ac.in}

\cortext[cor]{Corresponding author}
\fntext[fn]{Supported by DST-SERB project MTR/2017/000797.}

\address[NITM]{Dept. of Mathematics, NIT Meghalaya, Shillong 793003, India}





\begin{abstract}
In a separable Hilbert space $\mathcal H$, two frames $\{f_i\}_{i \in I}$ and $\{g_i\}_{i \in I}$ are said to be woven if there are constants $0<A \leq B$ so that for every $\sigma \subset I$, $\{f_i\}_{i \in \sigma} \cup \{g_i\}_{i \in \sigma ^c}$ forms a frame for $\mathcal H$ with the universal bounds $A, B$. This article provides methods of constructing  woven frames. In particular, bounded linear operators are used to create woven frames from a given frame.  Several examples are discussed to validate the results. Moreover, the notion of woven frame sequences is introduced and characterized.
\end{abstract}


\end{frontmatter}

%\linenumbers

\section{Introduction}	Hilbert space frame was first initiated by D. Gabor \cite{Ga46} in 1946 to reconstruct signals using fourier co-efficients.  Later, in 1986, frame theory was reintroduced and popularized by Daubechies, Grossman and Meyer \cite{DaGrMa86}.

% Since then frame theory has been widely used by mathematicians and engineers in various fields of mathematics and engineering sciences, namely, operator theory \cite{HaLa00}, harmonic analysis \cite{Gr01}, coding theory \cite{HoPa04}, signal processing \cite{Fe99}, sensor network \cite{CaKuLiRo07}, data analysis \cite{CaKu12}, etc.

Frame theory literature became richer through several generalizations, namely,  $G$-frame (generalized frames) \cite{ Su06}, $K$-frame (frames for operators (atomic systems)) \cite{Ga12},  fusion frame (frames of subspaces) (\cite{CaKu04, KhAs07}), $K$-fusion frame (atomic subspaces) \cite{Bh17},  etc. and some spin-off applications  by means of  Gabor analysis in (\cite{EaPa12, EaPa13}), dynamical system in mathematical physics in \cite{MiCh18}, nature of shift invariant spaces on the Heisenberg group in \cite{ArRa19}, characterizations of discrete wavelet frames in $\mathbb{C}^{\mathbb N}$ in \cite{DeVa17}, extensions of dual wavelet frames in \cite{EaPa18}, constructions of disc wavelets in \cite{AbGi14}, orthogonality of frames on locally compact abelian groups in \cite{GaSh19}  and many more.

% these generalizations have been proved to be useful in various applications.

%Also there are some spin-off applications of frame theory by means of robustness to erasures in \cite{ChLvSu18},  \cite{CoKiLoMaNaSh13}, \cite{NaRaRiToWoZi11}, through which frame theory became much more prosperous in theoretical sciences.

Let us consider a scenario: suppose in a sensor network system, there are sensors $A_1, A_2, \cdots, A_n$ which capture data to produce certain results. These sensors can be characterized by frames. In case one of these sensors, say $A_k$, fails to operate due to some technical reason, then the results obtained from these sensors may contain errors. Now assume that there are another set of sensors $B_1, B_2, \cdots, B_n$ which does play similar role as $A_i$'s. In addition, in the case of $A_k$ fails, $B_k$ can substitute so that obtained results are error free. Such an intertwinedness between two sets of sensors, or in general between two frames, leads to the idea of weaving frames. Weaving frames or woven frames were recently introduced by Bemrose et. al. in \cite{BeCaGrLaLy16}. After that Deepshikha et. al. produced a generalized form of weaving frames in \cite{DeVaVe17}, they also studied the weaving properties of generalized continuous frames in \cite{VaDe16}, vector-valued (super) weaving frames in \cite{DeVa18}.

This article focuses on study, characterize and explore several properties of woven frames. The outline of this article is organized as follows. Section \ref{Sec-Preli} is devoted to the basic definitions and results related to various kinds of frames, angle and gap between  subspaces. Moreover, the characterizations of woven frames are analyzed in Section \ref{Sec-char}. Finally, woven frame sequences are established in Section \ref{Sec-wov seq}.

Throughout the paper, $\mathcal{H}$ is a separable Hilbert space. We denote by $\mathcal H^n$ an $n$-dimensinal Hilbert space, $\{e_i\}_{ i \in [n]},~e_i=(\delta_{i, k})_{k\geq 1}$ is an orthonormal basis in $\mathcal H^n$, $\mathcal L(\mathcal H^n)$ to be a collection of all bounded, linear operators on $\mathcal H^n$, $R(T)$ is denoted as the range of the operator $T$, by $\hat {\delta}(M,N)$ we denote the {\it gap} between two closed subspaces $M$ and $N$  of a Hilbert space $\mathcal H$, $c_{0} (M,N)$ is denoted as the {\it cosine} of the minimal angle between $M$ and $N$,  $[n]=\lbrace 1,2,\cdots,n \rbrace$ and the index set $I$ is either finite or countably infinite. Given $J\subset I$, a Bessel sequence $\{f_i\}_{i\in I}$ and $\{c_i\}\in l^2(I)$, we define $T_{J_f} (\{c_i\}) = \sum\limits_{i\in J} c_if_i$. It is to be noted that $T_{J_f}$ is well-defined.

\section{Preliminaries}\label{Sec-Preli}

In this section we recall basic definitions and results needed in this paper. For more details we refer the books  written by  Casazza and Kutyniok \cite{CaKu12} and Ole Christensen \cite{Ch08}.

\subsection{Frame} A collection  $\{ f_i \}_{i\in I}$ in $\mathcal{H}$ is called a \emph{frame} for $\mathcal H$ if there exist constants $A,B >0$ such that \begin{equation}\label{Eq:Frame} A\|f\|^2~ \leq ~\sum_{i\in I} |\langle f,f_i\rangle|^2 ~\leq ~B\|f\|^2,\end{equation}for all $f \in \mathcal{H}$. The numbers $A, B$ are called \emph{frame bounds}. The operator $S: \mathcal H\rightarrow \mathcal H$, defined by  $Sf = \sum_{i\in I} \langle f, f_i\rangle f_i$ is called the frame operator for $\{f_i\}_{i \in I}$. It is well-known that the frame operator is linear, bounded, positive, self-adjoint and invertible.

%{\bf{Reconstruction formula}}: Every element in $\mathcal{H}$ can be represented using frame elements as follows:
%\begin{equation}\label{Eq:Frame-recons}
%f = \sum_{i\in I} \langle f, S^{-1}f_i\rangle f_i = \sum_{i\in I} \langle f, f_i\rangle S^{-1}f_i
%\end{equation}
%Since the frame elements are not necessarily linearly independent, this representation is not unique, in general.

\begin{defn}
%\begin{enumerate}
Let $\{ f_i \}_{i\in I}$ and $\{ g_i \}_{i\in I}$  be two frames for $\mathcal H$. If for all $f \in \mathcal H$, $f = \sum \limits_{i\in I} \langle f, g_i\rangle f_i $, then $\{ g_i \}_{i\in I}$ is called a dual frame of $\{ f_i \}_{i\in I}$. If $S$ is the frame operator of $\{f_i\}_{i\in I}$, then $ \{ S^{-1}  f_i \}_{i\in I}$ is said to be the {\it canonical dual} frame of $\{ f_i \}_{i\in I}$.

%\end{enumerate}

\end{defn}

\begin{prop} (\cite{Ch08, HaKoLaWe07}) \label{prop-span} A finite family $\lbrace f_i \rbrace_{i \in [m]}$ in $\mathcal H^n$, forms a frame for  $\mathcal H^n$ if and only if $span \lbrace f_i \rbrace_{i \in [m]}= \mathcal H^n$.

\end{prop}

\subsection{Woven and full spark frame}  In a Hilbert space $\mathcal H$, a family of frames $\lbrace f_{ij} \rbrace_{i \in \mathbb N , j \in [M]}$ is said to be {\it weakly woven} if for any partition $\lbrace \sigma_j \rbrace_{j \in [M]}$ of $\mathbb N$, $\lbrace f_{ij} \rbrace_{i \in \sigma_j , j \in [M]}$ forms a frame for $\mathcal H$.

Also, in $\mathcal H$, two frames $\mathcal F=\lbrace f_i \rbrace_{i \in I}$ and  $\mathcal G=\lbrace g_i \rbrace_{i \in I}$ are said to be {\it woven} if for every $\sigma \subseteq I$,  $\lbrace f_i \rbrace_{i \in \sigma} \cup \lbrace g_i \rbrace_{i \in \sigma^c}$ also forms a frame for $\mathcal H$ and the associated frame operator for every weaving is defined as,
$$S_{\mathcal F \mathcal G}f=\sum_{i \in \sigma}\langle f,f_i \rangle f_i + \sum_{i \in \sigma^c}\langle f,g_i \rangle g_i, ~~ \text{for}~ \text{all}~f \in \mathcal H.$$

\begin{thm} \cite{BeCaGrLaLy16} In $\mathcal H$, two frames are weakly woven if and only if they are woven.

\end{thm}

Moreover, a frame with $m$ elements in $\mathcal H^n$, is said to be a {\it full spark frame} if every subset of the frame, with cardinality $n$, is also a frame for $\mathcal H^n$. For example, if $\{e_i\}_{ i \in [2]},~e_i=(\delta_{i, k})_{k\geq 1}$, is an orthonormal basis in $\mathbb R^2$, then $\{e_1, e_2, e_1 + e_2\}$  is a full spark frame for $\mathbb R^2$. Furthermore, if every element of a finite frame can be represented as a linear combination of the remaining others, then the frame is called a {\it weak full spark frame}.

For example, if $\{e_i\}_{ i \in [3]}$ is an orthonormal basis of $\mathbb R^3$, $\lbrace e_1, e_1, e_2, e_2, e_3, e_3  \rbrace$ is a weak full spark frame but not a full spark frame. In this context, it is a fortuitous evident that every nontrivial (other than exact) full spark frame is also a weak full spark frame.

\begin{prop} \cite {BeCaGrLaLy16} \label{Bessel} Let $\{f_{ij}\}_{i \in I}$ be a collection of Bessel sequences in $\mathcal H$ with bounds $B_j$'s for every $j \in [m]$, then every weaving forms a Bessel sequence with bound $\sum \limits_{j \in [m]} B_j$ and norm of corresponding synthesis operator is bounded by $\sqrt{\sum \limits_{j \in [m]} B_j}$.

\end{prop}

\begin{prop} \label{frame sequence} Every frame sequence is a Bessel sequence.

\end{prop}

\begin{proof} Let $\{f_i\}_{i \in I}$ be a frame sequence in $\mathcal H$, then it forms a frame for $F=\overline {\text{span}}\{f_i\}_{i \in I}$. Therefore $\mathcal H = F \oplus F ^{\perp}$ and hence for every $f \in \mathcal H$ we have $f = f_F + f_{F^{\perp}}$. Therefore for some $B>0$ we obtain,
$$\sum_{i \in I} |\langle f,f_i \rangle |^2=\sum_{i \in I} |\langle f_F,f_i \rangle |^2 \leq B \|f_F\|^2 \leq B\|f\|^2.$$

\end{proof}

\begin{remark}
The converse implication of the above proposition is not necessarily true, which is evident from the following fact:\\
Consider $\{e_i\}_{i \in \mathbb N}$ as an orthonormal basis for $\mathcal H$ and let us define
$$f_i = e_i + e_{i+1}, ~i \in \mathbb N.$$
Then $\{f_i\}_{i \in \mathbb N}$ forms a Bessel sequence in $\mathcal H$ but not a frame sequence. For detail discussion regarding the same we refer the Example $5.1.10$ in \cite{Ch08}.
\end{remark}

\begin{lem} \label{woven-equivalency} Let $\lbrace f_i \rbrace_{i \in I}$ and  $\lbrace g_i \rbrace_{i \in I}$ be frames for $\mathcal H$ with bounds $A_1, B_1$ and $A_2, B_2$, respectively. Then the following results are equivalent:
\begin{enumerate}
\item $\lbrace f_i \rbrace_{i \in I}$ and  $\lbrace g_i \rbrace_{i \in I}$ are woven.
\item For every $\sigma \subset I$, if $S_\sigma$ is the associated frame operator for the corresponding weaving, then for every $f \in \mathcal H$ we have $\|S_\sigma f\| \geq k\|f\|$ for some $k>0$, independent of $\sigma$.
\end{enumerate}
\end{lem}

\begin{proof} \underline{$(1 \implies 2)$}

Let $\lbrace f_i \rbrace_{i \in I}$ and  $\lbrace g_i \rbrace_{i \in I}$ be woven with universal frame bounds $A, B$. Therefore for every $\sigma \subset I$ and for every $f \in \mathcal H$ we have,
$$A\|f\|^2 \leq \sum_{i \in \sigma}|\langle f,f_i \rangle |^2 + \sum_{i \in \sigma^c}|\langle f,g_i \rangle |^2 \leq B\|f\|^2.$$
If $S_\sigma$ is the associated frame operator for the corresponding weaving, then from the above inequality we have,
\begin{equation*} \label{equation}
A\|f\|^2 \leq \langle S_\sigma f,f \rangle \leq B\|f\|^2.
\end{equation*}
Therefore,
\begin{eqnarray*}
\|S_\sigma f\|=\text{sup}_{\|g\|=1} |\langle S_\sigma f,g \rangle| \geq  \left \langle S_\sigma f, \frac {f} {\|f\|} \right \rangle \geq A\|f\|.
\end{eqnarray*}

\underline{$(2 \implies 1)$}

For every $f \in \mathcal H$ and $\sigma \subset I$, $S_\sigma f=T_\sigma T_\sigma^*f$, where $T_\sigma, T_\sigma^*$ are the corresponding synthesis and analysis operators, respectively; and we have
$k^2\|f\|^2 \leq \|S_\sigma f\|^2=\|T_\sigma T_\sigma^*f\| \leq \|T_\sigma \|^2 \|T_\sigma^*f\|^2$ and hence we obtain,
$$\frac {k^2} {B_1+B_2} \|f\|^2 \leq \|T_\sigma^*f\|^2 = \sum_{i \in \sigma}|\langle f,f_i \rangle |^2 + \sum_{i \in \sigma^c}|\langle f,g_i \rangle |^2.$$
The universal upper frame bound for the weaving will be achieved by Proposition \ref{Bessel}.

\end{proof}

\begin{thm} \cite {BeCaGrLaLy16} \label{thm-span} In $\mathcal H^n$, two frames $\lbrace f_i \rbrace_{i \in [m]}$ and  $\lbrace g_i \rbrace_{i \in [m]}$ are woven if and only if for every $\sigma \subseteq [m]$, $\text{span}(\lbrace f_i \rbrace_{i \in \sigma} \cup \lbrace g_i \rbrace_{i \in \sigma^c}) = \mathcal H^n$.
\end{thm}

\subsection{{Gap and angle between subspaces}}  Let $M$ and $N$ be two closed subspaces of a Hilbert space $\mathcal H$. Then the {\it gap} between $M$ and $N$ is given by, $\hat {\delta}(M,N)= \max \lbrace \delta(M,N),\delta(N,M)\rbrace$, where $\delta(M,N)= \sup \limits_{x \in S_{M}} dist(x,N)$,  $S_{M}$ is the unit sphere in $M$ and $dist(x,N)$ is the distance from $x$ to $N$.

Again the $cosine$ of the angle between two closed subspaces $M$ and $N$ of  a Hilbert space $\mathcal H$ is given by,$$c(M,N)=sup\lbrace |\langle x,y \rangle|: x \in M \cap (M \cap N)^{\perp}, \| x\| \leq 1,  y \in N \cap (M \cap N)^{\perp}, \| y\| \leq 1 \rbrace$$ and the $cosine$ of the minimal angle of the same  is given by, $$c_{0} (M,N)=sup\lbrace |\langle x,y \rangle|: x \in M, \| x\| \leq 1,  y \in N , \| y\| \leq 1 \rbrace.$$

 For the extensive discussion regarding the gap and the angle between two subspaces, we refer (\cite{BaTr09, De95, Ka80}).

\begin{remark} \cite{Ka80} \label{gap}  Let $M$ and $N$ be two closed subspaces of a Hilbert space $\mathcal H$. Then the followings are satisfied:
\begin{enumerate}
\item $\delta(M,N)=0$ if and only if $M \subset N$.
\item $\hat {\delta}(M,N)=0$ if and only if $M=N$.
\end{enumerate}
\end{remark}

\begin{lem} \cite{De95} \label{angle-perp}  Let $M$ and $N$ be two closed subspaces of a Hilbert space $\mathcal H$. Then $c_{0} (M,N) =0$ if and only if $M \perp N$.

\end{lem}

\begin{thm} \cite{De95} \label{angle}  Let $M$ and $N$ be two closed subspaces of a Hilbert space $\mathcal H$. Then the following statements are equivalent:
\begin{enumerate}
\item $c_{0} (M,N)<1$ .
\item $M \cap N=\lbrace 0 \rbrace$ and $M+N$ is closed.
\end{enumerate}
\end{thm}

\begin{thm}(Douglas' factorization theorem \cite{Do66})\label{Thm-Douglas} Let $\mathcal{H}_1, \mathcal{H}_2,$ and $\mathcal H$ be Hilbert spaces and $S\in \mathcal L(\mathcal H_1, \mathcal H)$, $T\in \mathcal L(\mathcal H_2, \mathcal H)$. Then the following results are equivalent:
\begin{enumerate}
\item $R(S) \subseteq R(T)$.
\item $SS^* \leq \alpha TT^*$ for some $\alpha >0$.
\item $S = TL$ for some $L \in \mathcal L(\mathcal H_1, \mathcal H_2)$.
\end{enumerate}
\end{thm}

\section{Characterization of Woven Frames}\label{Sec-char}
In this section, we characterize woven frames, mainly through constructions of frames from  given frames. The proposed constructions are based on the images of a given frame by means of  bounded linear operators. Before diving into the main results, we start the discussion with the following Proposition.

\begin{prop} Let $\lbrace f_i \rbrace_{i \in [m]}$ be a frame for $\mathcal H^n$. Suppose $f_{m+1}=0$, then $\lbrace (f_i-f_{i+1}) \rbrace_{i \in [m]}$ is also a frame for $\mathcal H^n$ and these two frames are woven.
\end{prop}

\begin{proof} The proof will be followed from elementary row operations.
\end{proof}

\begin{remark} In the above Proposition instead of  $f_{m+1} = 0$, if $f_{m+1} = f_1$, then $\lbrace (f_i-f_{i+1}) \rbrace_{i \in [m]}$ may not be a frame for $\mathcal H^n$. For example, let  $\{e_i\}_{ i \in [3]}$ be an orthonormal basis in $\mathbb R^3$, then $\{-e_1 + e_2, e_1 + e_2, -2 e_1 + e_2 - e_3, e_1 + e_2 + e_3\}$ is a frame for $\mathbb R^3$. But clearly, $\lbrace (f_i-f_{i+1}) \rbrace_{i \in [4]} = \{-2 e_1, 3 e_1 + e_3, -3 e_1 - 2 e_3, 2 e_1 + e_3\}$ is not a frame for $\mathbb R^3$.

But if this is a frame, then they must be woven, which is evident from the fact that $f=\sum\limits_{i\in [m]} b_i \left(f_i-f_{i+1}\right)$ can be written as $f=(b_1 - b_j)f_1+(b_2 - b_1)f_2+...(b_j - b_{j-1})f_j + (b_{j+1} - b_j)(f_{j+1} - f_{j+2})+...(b_{m-1} - b_j)(f_{m-1} - f_m)+(b_m - b_j)(f_m - f_{m+1})$.
\end{remark}

\begin{remark} If $\lbrace f_i \rbrace_{i \in [m]}$ is a frame for $\mathcal H^n$ and suppose $f_{m+1} = 0$, then $\lbrace (\alpha f_i + \beta  f_{i+1}) \rbrace_{i \in [m]}$, $\alpha, \beta \neq 0$, is also a frame for $\mathcal H^n$ and they are woven.
\end{remark}

In the following result, we present conditions under which image of a given frame under an idempotent operator is woven with the said frame.

\begin{lem} Let $F (\neq 0) \in \mathcal L(\mathcal H)$ be a closed range,  idempotent operator with $R(F) = R(F^*)$. Suppose $\lbrace f_i \rbrace_{i \in I}$ is a frame for $R(F^*)$, then $\lbrace F f_i \rbrace_{i \in I}$ is also a frame for $R(F^*)$ and they are woven.
\end{lem}

\begin{proof} Let $\lbrace f_i \rbrace_{i \in I}$ be a frame for $R(F^*)$ with bounds $A, B$. Then for every $f \in  R(F^*)$ we have,
$$\sum_{i \in I} | \langle f, F f_i \rangle |^2 = \sum_{i \in I} | \langle F^*f,  f_i \rangle |^2 \leq B\|F^* f\|^2 \leq B\|F\|^2 \|f\|^2.$$
Again since $f \in  R(F^*)=R(F)$, $\|f\|^2 = \|(F^*)^{\dag} F^* f\|^2 \leq \|(F^*)^{\dag}\|^2 \|F^* f\|^2$ and hence
$\frac {\|f\|^2} {\|(F^*)^{\dag}\|^2} \leq \|F^* f\|^2$. Therefore for every $f \in  R(F^*)$ we obtain,
$$\sum_{i \in I} | \langle f, F f_i \rangle |^2 = \sum_{i \in I} | \langle F^*f,  f_i \rangle |^2  \geq A \|F^* f\|^2 \geq \frac {A} {\|(F^*)^{\dag}\|^2} \|f\|^2.$$
Consequently, $\lbrace F f_i \rbrace_{i \in I}$ forms a frame for $R(F^*)$.

Moreover, for every $f \in  R(F^*)$, there exists $g \in \mathcal H$ such that $F^* g=f$ and since $F^2 = F$, for every $\sigma \subset I$ and for all $f \in  R(F^*)$  we obtain,
\begin{eqnarray*}
\sum_{i \in \sigma} | \langle f, f_i \rangle |^2 + \sum_{i \in \sigma ^c} | \langle f, F f_i \rangle |^2 & = & \sum_{i \in \sigma} | \langle F^* g, f_i \rangle |^2 + \sum_{i \in \sigma ^c} | \langle F^* g, F f_i \rangle |^2
\\ & = & \sum_{i \in \sigma} | \langle  g, F f_i \rangle |^2 + \sum_{i \in \sigma ^c} | \langle  g, F f_i \rangle |^2
\\ & = & \sum_{i \in I} | \langle  g, F f_i \rangle |^2
\\ & = & \sum_{i \in I} | \langle  f,  f_i \rangle |^2.
\end{eqnarray*}

\noindent Therefore, our assertion is tenable.
\end{proof}

\begin{remark} It is to be noted that, if one of the conditions of $F^2=F$ and $R(F) = R(F^*)$ fails, then the conclusion of the above Lemma may not hold. This is evident from the following two examples.
\end{remark}

\begin{ex}
Consider an idempotent operator $F$ on $\mathbb R^2$ so that $F e_1 = e_1 + 2 e_2, ~  F e_2 = 0$, then $R(F) =\text{~span~} \{e_1 + 2 e_2\} \neq \text{~span~} \{e_1\} = R(F^*)$.
 Now for the frame  $\mathcal F =\{e_1, e_2, e_1 + e_2\}$ for $\mathbb R^2$, $F(\mathcal F)= \{e_1 + 2 e_2, 0, e_1 + 2 e_2\}$  is not a frame for $R(F^*)$.
\end{ex}

\begin{ex}
Consider an operator $F$ on $\mathbb R^3$ so that $F e_1 = e_1 + e_2, ~ F e_2 = -  e_1 + e_2, ~ F e_3 = 0$, then $F^2 \neq F$ but $R(F)=R(F^*)$. Now let us choose a frame  $\{f_i\}_{i \in [3]} = \{e_1, e_1 - e_2, 2 e_1\}$ for $R(F^*)$, then
$\lbrace F f_i \rbrace_{i \in [3]} =\{e_1 + e_2, 2 e_1, 2 e_1 + 2 e_2\}$ is also a frame for $R(F^*)$, but they are not woven, which can be verified for $\sigma = \{1 , 3\}$.
\end{ex}

In the following outcomes we study  woven-ness of frames and their  images  under invertible operators.

\begin{remark} The image of a frame under invertible operators is not necessarily woven with the frame, which is evident from the following example:\\
$\mathcal F=\{e_2, e_1 + e_2, 2e_2 \}$ is a frame for $\mathbb R^2$. Let us consider an invertible operator $T$ so that $Te_1=e_1 - e_2,~ Te_2=-e_1 - e_2$. Then $T \mathcal F=\{-(e_1 + e_2), - 2e_2,- 2(e_1 + e_2) \}$ is also a frame for $\mathbb R^2$, however they are not woven, which can be verified by considering $\sigma = \{1, 3\}$.
\end{remark}

%\begin{prop} If $\lbrace f_i \rbrace_{i \in [m]}$ is a frame for $\mathcal H^n$ with the associated frame operator $S$, then $\lbrace f_i \rbrace_{i \in [m]}$ and $\lbrace S f_i \rbrace_{i \in [m]}$ are woven.
%\end{prop}

%\begin{proof} Since $S$ is the frame operator of $\lbrace f_i \rbrace_{i \in [m]}$, for all $f \in \mathcal H^n$ we have,
%$$S f = \sum_{j \in [m]} \langle f,f_j \rangle f_j ~~.$$
%Let $\sigma = \lbrace 1, 2,..., k \rbrace \subset [m]$. Therefore we obtain,
%\begin{eqnarray*}
%\text{span~}  (\lbrace f_i \rbrace_{i \in \sigma} \cup \lbrace S f_i \rbrace_{i \in \sigma^c}) &=& \text{~span~}  \left (\lbrace f_i %\rbrace_{i \in \sigma} \cup \left \lbrace \sum_{j \in [m]} \langle f_i,f_j \rangle f_j \right \rbrace_{i \in \sigma^c} \right)\\
%&=&\text{~span~} \lbrace f_i \rbrace_{i \in [m]} ~=~ \mathcal H^n.\end{eqnarray*}
%Hence our goal is executed.

%\end{proof}

%In the following proposition we  characterize full spark frame by means of  invertible operator.

\begin{prop} \label{inv-woven}The image of woven frames under invertible operator preserves their woven-ness.
\end{prop}

\begin{proof} Let $\lbrace f_i \rbrace_{i \in I}$ and $\lbrace g_i \rbrace_{i \in I}$ be woven in $\mathcal H$ with universal bounds $A, B$ and suppose $T \in \mathcal L(\mathcal H)$ is an invertible operator. Then $\lbrace T f_i \rbrace_{i \in I}$ and $\lbrace T g_i \rbrace_{i \in I}$ are also frames for $\mathcal H$.

For every $\sigma \subset I$ and for every $f \in \mathcal H$ we obtain,
\begin{eqnarray*}
\sum_{i \in \sigma} | \langle f, Tf_i \rangle |^2 + \sum_{i \in \sigma ^c} | \langle f, T g_i \rangle |^2 &=& \sum_{i \in \sigma} | \langle T^*f, f_i \rangle |^2 + \sum_{i \in \sigma ^c} | \langle T^*f, g_i \rangle |^2
\\& \geq & A\|T^*f\|^2
\\& \geq & \frac {A} {\|T^{-1}\|^2}\|f\|^2.
\end{eqnarray*}
The upper frame bound of the respective weaving will be achieved from the Proposition \ref{Bessel}.
\end{proof}

In the following theorem we discuss a necessary and sufficient condition of woven frames.

\begin{thm} \label{iff woven} Let $\mathcal F=\lbrace f_i \rbrace_{i \in I}$ and $\mathcal G=\lbrace g_i \rbrace_{i \in I}$ be two frames for $\mathcal H$. Then they are woven if and only if $R(T_{\mathcal F \mathcal G})=\mathcal H$, where $T_{\mathcal F \mathcal G}$ is the associated synthesis operator of the respective weaving.
\end{thm}

\begin{proof} Since $\mathcal F$ and $\mathcal G$ are frames for $\mathcal H$, by Proposition \ref{Bessel} every weaving is a Bessel sequence and hence $T_{\mathcal F \mathcal G}$ is well-defined.

 Let $R(T_{\mathcal F \mathcal G})=\mathcal H = R(I_{\mathcal H})$. Therefore using Theorem \ref{Thm-Douglas}, there exists $A>0$ and  for every $f \in \mathcal H$ we have $\|T_{\mathcal F \mathcal G}^* f\|^2 \geq A\|f\|^2$ and hence for every $\sigma \subset I$ we obtain,
$$\sum_{i \in \sigma} | \langle f, f_i \rangle |^2 + \sum_{i \in \sigma ^c} | \langle f,  g_i \rangle |^2 \geq  A\|f\|^2.$$
The converse implication will be followed from the definition of woven frame.
\end{proof}

\begin{ex} $\mathcal F=\{e_1 + e_2, e_1 +2 e_2, e_1 - e_2\}$ is a frame for $\mathbb R^2$. If $S$ is its frame operator, then $S\mathcal F=\{5e_1 + 8e_2, 7e_1 +14 e_2, e_1 - 4e_2\}$. It is to be noted that the associated synthesis operators of every weaving are onto.
%Therefore using theorem (\ref{iff woven}), $\mathcal F$ and $S\mathcal F$ are woven.
\end{ex}

Since invertible operators preserve linear independency of vectors, so it is a natural intuition that a finite frame is woven with its image under the associated frame operator.

\noindent \underline{\bf Problem 1:} If $\{f_i\}_{i \in I}$ is a frame for $\mathcal H$ with the associated frame operator $S$, Can $\{f_i\}_{i \in I}$ and $\{Sf_i\}_{i \in I}$ woven?

At this moment, we are impotent to deliver an affirmative response, although we strongly believe that the same can be executed in this context. If so, then using Proposition \ref{inv-woven}, it is evident that $\{f_i\}_{i \in I}$ and $\{S^{-1}f_i\}_{i \in I}$ are woven.

\noindent \underline{\bf Problem 2:} Whether a frame is woven with its dual?

\section{Woven Frame Sequence}\label{Sec-wov seq}

A family $\lbrace f_i \rbrace_{i \in I}$ in $\mathcal H$ is said to be a frame sequence if it forms a frame  for its closed, linear span. It is to be noted that $\{f_i\}_{i\in I}$ is not necessarily a frame for $\mathcal H$. In this section we explore the possibilities of two frame sequences together, through the concept of woven frames, form a frame for $\mathcal H$.

\begin{defn} Two frame sequences  $\mathcal F= \lbrace f_i \rbrace_{i \in I}$ and $\mathcal G = \lbrace g_i \rbrace_{i \in I}$  in $\mathcal H$, are said to be {\bf woven frame sequences}, if for every $\sigma \subset I$, $\lbrace f_i \rbrace_{i \in \sigma} \cup \lbrace g_i \rbrace_{i \in \sigma^c} $ forms a frame for $\mathcal H$.
\end{defn}

\begin{ex}
For example, In $\mathbb R^3$, the frame sequences $\{e_1 + 2 e_3, e_1 - e_3, -e_1 + 2 e_3, e_1 + 3 e_3\}$ and $\{e_1 -  e_2, e_1 + 2 e_2, -e_1 + 3 e_2, e_1 - 2 e_2\}$ are woven  whereas $\{e_1 + 2 e_3, e_1 - e_3, -e_1 + 2 e_3, e_1 + 3 e_3\}$ and $\{e_1 -  e_2, e_1 + 2 e_2, -e_1 + 3 e_2, e_1 \}$ are not.
  \end{ex}

%In this section we define woven frame sequences, through which we introduce that every weaving of two given frame sequences, forms a frame for the associated space.

%An interesting observation is, each weaving of the assigned frame sequences, is a frame for the corresponding space, though originally they are not even be frame for the related Hilbert space.

The notion of woven frame sequences is beneficial for its practical importance, because instead of two given frames, if we consider two frame sequences, then less restriction is there in our primary assumption and due to this fact, it is cost-effective.

%In the next lemma we produce some characterizations  on woven as well as not woven frame sequences.

\begin{thm} Let $\mathcal F=\lbrace f_i \rbrace_{i \in [m]}$ and $\mathcal G=\lbrace g_i \rbrace_{i \in [m]}$ be two frame sequences in  $\mathcal H^n$. Then the following statements are satisfied :
\begin{enumerate}
\item   $\mathcal F$ and  $\mathcal G$ are not woven if there exists a non-trivial $\sigma \subset [m]$ so that $c_0 \lbrace span (\mathcal F_{\sigma} \cup \mathcal G_{\sigma^c}), span (\mathcal F_{\sigma^c} \cup \mathcal G_{\sigma}) \rbrace<1~.$

\item If  for every non-trivial $\sigma \subset [m]$, $\hat {\delta} \lbrace span (\mathcal F_{\sigma} \cup \mathcal G_{\sigma^c}), span (\mathcal F_{\sigma^c} \cup \mathcal G_{\sigma}) \rbrace=0$ and $c_0 \lbrace (span (\mathcal F_{\sigma} \cup \mathcal G_{\sigma^c})), (span (\mathcal F_{\sigma^c} \cup \mathcal G_{\sigma}))^c \rbrace=0=c_0 \lbrace (span (\mathcal F_{\sigma^c} \cup \mathcal G_{\sigma})) , (span (\mathcal F_{\sigma} \cup \mathcal G_{\sigma^c}))^c$, then $\mathcal F$ and  $\mathcal G$ are  woven.
\end{enumerate}
\end{thm}

\begin{proof} Using  Lemma  \ref{angle-perp} and  Theorem \ref{angle}~, our assertions are quickly  plausible.

\end{proof}

%In the following result we further characterize woven frame sequence by means of gap between two closed subspaces.

\begin{thm} In $\mathcal H^n$, if $\mathcal F=\lbrace f_i \rbrace_{i \in [m]}$ and $\mathcal G=\lbrace g_i \rbrace_{i \in [m]}$ are two woven frame sequences, then for every non-trivial $\sigma \subset [m]$,
$$\hat {\delta} \lbrace span (\mathcal F_{\sigma} \cup \mathcal G_{\sigma^c}), span (\mathcal F_{\sigma^c} \cup \mathcal G_{\sigma}) \rbrace=0~.$$

\end{thm}

\begin{proof} If $\mathcal F$ and $\mathcal G$ are woven, then for every non-trivial $\sigma \subset [m]$, both $\mathcal F_{\sigma} \cup \mathcal G_{\sigma^c}$ and $\mathcal F_{\sigma^c} \cup \mathcal G_{\sigma}$ constitute frames for $\mathcal H^n$.
Hence the conclusion directly follows from the Remark \ref{gap}~.
\end{proof}

\begin{remark} It is to be noted that, the two foregoing outcomes also hold for characterizing woven frames.
\end{remark}

%The following Theorem characterizes woven frame sequence through the concept of $K$-frame. A similar idea can be found in %\cite{DeVa18}.

%\begin{thm} For every $K \in \mathcal L(\mathcal H)$, every family of woven frame sequences is a family of woven $K$-frames. %Moreover, if $R(K)$ is closed then it is a family of woven frames for $R(K)$.

%\end{thm}

%\begin{proof} Let $\{\{ f_{ij}\}_{j \in \mathbb N}: i \in [m]\}$ be a family of woven frame sequences in $\mathcal H$ with the %universal bounds $A, B$. Consider a partition $\{\sigma_i\}_{i \in [m]}$. Then for every $f \in \mathcal H$ we get,
%$$\sum_{i \in [m]} \sum_{j \in \sigma_i} |\langle f,K(f_{ij}) \rangle|^2 = \sum_{i \in [m]} \sum_{j \in \sigma_i} |\langle K^*f,f_{ij} %\rangle|^2 \leq B\|K^*f\|^2 \leq B \|K\|^2 \|f\|^2.$$

%On the otherhand for all $f \in \mathcal H$ we obtain,
%$$A \|K^*f\|^2 \leq \sum_{i \in [m]} \sum_{j \in \sigma_i} |\langle K^*f,f_{ij} \rangle|^2 = \sum_{i \in [m]} \sum_{j \in \sigma_i} |%\langle f,K(f_{ij}) \rangle|^2 .$$
%Therefore our first assertion is established.

%Again if $R(K)$ is closed then for every $f \in R(K)$ we have,
%$$\frac {A} {\|K^{* \dag}\|^2} \|f\|^2 \leq A \|K^*f\|^2 \leq \sum_{i \in [m]} \sum_{j \in \sigma_i} |\langle K^*f,f_{ij} \rangle|^2 %= \sum_{i \in [m]} \sum_{j \in \sigma_i} |\langle f,K(f_{ij}) \rangle|^2 .$$
%Hence our second assertion is tenable.
%\end{proof}

%\begin{remark} Can we produce a necessary and sufficient condition on two frame sequences to be woven?
%\end{remark}

In the following results we explore sufficient conditions for woven-ness between frame and frame sequence. The following theorem shows that woven-ness is preserved under perturbation.

\begin{thm} Let  $\mathcal F = \{f_i\}_{i \in I}$, $\mathcal G = \{g_i\}_{i \in I}$ be two woven frames for $\mathcal H$ with the universal frame bounds $A, B$ and $T_{\mathcal G}$ be the corresponding synthesis operator of $\mathcal G$. If $H=\{h_i\}_{i \in I}$ is a frame sequence in $\mathcal H$ with the associated synthesis operator $T_{H}$ so that $(\|T_{\mathcal G}\| + \|T_{H}\|) \|T_{\mathcal G} - T_{H}\| < A$, then $\mathcal F$ and $H$ are woven.

\end{thm}

\begin{proof}
 %Let for every $\sigma \subset I $ and every $\{c_i\}_{i \in I} \subset l^2(I)$, $T_{\sigma_f}(\{c_i\}) = \sum \limits_{i \in \sigma} c_i f_i$, $T_{\sigma^c_f}(\{c_i\}) = \sum \limits_{i \in \sigma^c} c_i f_i$, $T_{\sigma_g}(\{c_i\}) = \sum \limits_{i \in \sigma} c_i g_i$, $T_{\sigma^c_g}(\{c_i\}) = \sum \limits_{i \in \sigma^c} c_i g_i$.

For every $\sigma\subset I$, let $P_{\sigma}$ be the orthogonal projection on $\overline{\text{span}} \{e_i\}_{i \in \sigma}$ and therefore $T_{\sigma_f}=T_{\mathcal F} P_{\sigma}$.
%and hence $\|T_{\sigma^c_f} - T_{\sigma^c_g}\| \leq \|T_{\mathcal F} - T_{\mathcal G}\|, ~\|T_{\sigma^c_f}\| \leq \|T_{\mathcal F}\|,~\|T_{\sigma^c_g}\| \leq \|T_{\mathcal G}\| ~\text{and}~ \|T_{\sigma^c_h}\| \leq \|T_{H}\|.$\\

%Therefore applying the Proposition (\ref{frame sequence}), above mentioned operators are well-defined.
\noindent Since $\mathcal F$ and $\mathcal G$ are woven with the universal bounds $A, B$, then using Lemma \ref{woven-equivalency}, for every $\sigma \subset I$ and every $f \in \mathcal H$ we have,
\begin{equation} \label{woven}
A\|f\| \leq \|\sum_{i \in \sigma} \langle f,f_i \rangle f_i + \sum_{i \in \sigma ^c} \langle f,g_i \rangle g_i \|.
\end{equation}
The proof will be completed with the following steps.\\
\noindent {\bf Step 1:} For all $f \in \mathcal H$ and $\sigma \subset I$,
$$\|\sum_{i \in \sigma ^c} \langle f,g_i \rangle g_i - \sum_{i \in \sigma ^c} \langle f,h_i \rangle h_i \| \leq (\|T_{\mathcal G}\| + \|T_{H}\|) \|T_{\mathcal G} - T_{H}\| \|f\|.$$

\noindent \underline{proof of Step 1:} Using Proposition \ref{frame sequence} and utilizing the  properties of the respective synthesis operators, for every $f \in \mathcal H$ and $\sigma \subset I$ we have,
\begin{eqnarray*}
\|\sum_{i \in \sigma ^c} \langle f,g_i \rangle g_i - \sum_{i \in \sigma ^c} \langle f,h_i \rangle h_i \|
 &\leq &  \|T_{\mathcal G}\| \| T^*_{\mathcal G} - T^*_{H}\| \|f\| + \|T_{\mathcal G} - T_{H}\| \| T^*_{H}\| \|f\|
\\ & = & (\|T_{\mathcal G}\| + \|T_{H}\|) \|T_{\mathcal G} - T_{H}\| \|f\|.
\end{eqnarray*}

\noindent {\bf Step 2:}  For every weaving, universal lower frame bound is $\frac {[A -  \|T_{\mathcal G} - T_{H}\| (\|T_{\mathcal G}\| + \|T_{H}\|)]^2} {B+B_1}$, where $B_1$ is an upper frame bound for $H$.

\noindent \underline{proof of Step 2:} Applying equation (\ref{woven}) and step 1, for every $f \in \mathcal H$ we obtain,
\begin{eqnarray*}
& \| & \sum  \limits_{i \in \sigma}  \langle f,f_i \rangle f_i + \sum_{i \in \sigma ^c} \langle f,h_i \rangle h_i \| \\ & \geq & \|\sum_{i \in \sigma} \langle f,f_i \rangle f_i + \sum_{i \in \sigma ^c} \langle f,g_i \rangle g_i \| - \|\sum_{i \in \sigma ^c } \langle f,g_i \rangle g_i - \sum_{i \in \sigma ^c} \langle f,h_i \rangle h_i \|
\\ & \geq & [A -  \|T_{\mathcal G} - T_{H}\| (\|T_{\mathcal G}\| + \|T_{H}\|)] \|f\|.
\end{eqnarray*}
Therefore the conclusion follows from Lemma \ref{woven-equivalency}.

\noindent The universal upper frame bound of the weaving will be achieved from the Proposition \ref{Bessel}.
\end{proof}

\begin{remark} If the frames $\mathcal F$, $\mathcal G$ are woven in $\mathcal H$ and $\mathcal G$, $H$ are woven in $\mathcal H$, then $\mathcal F$ and $H$ are not necessarily woven in $\mathcal H$, which is evident from the following example.

\end{remark}

\begin{ex} Let $\mathcal F = \{e_1, e_2, 2 e_1\}$, $\mathcal G = \{2 e_1, - e_2, - 2 e_2\}$ and $H=\{e_1, - e_1, 2 e_2\}$. Then $\mathcal F$ and $\mathcal G$ are woven as well as $\mathcal G$ and $H$ are woven , but $\mathcal F$ and $H$ are not woven in $\mathbb R^2$, as if we consider $\sigma = \{3\}$ then the associated weaving is $\{e_1, - e_1, 2 e_1\}$.

\end{ex}

\begin{cor} \label{woven frame sequence} Let $\mathcal F = \{f_i\}_{i \in I}$ be a frame for $\mathcal H$ with lower frame bound $A$ and $\mathcal G = \{g_i\}_{i \in I}$ be a frame sequence in $\mathcal H$. Let $T_{\mathcal F}$ and $T_{\mathcal G}$ be corresponding synthesis operators, respectively. Then $\mathcal F$ and $\mathcal G $ are woven if $$  (\|T_{\mathcal F} - T_{\mathcal G}\|) (\|T_{\mathcal F}\| + \|T_{\mathcal G}\|)< A.$$
\end{cor}

\begin{thm} Let  $\mathcal F = \{f_i\}_{i \in I}$ be a frame for $\mathcal H$ with frame bounds $A_1, B_1$ and $\mathcal G = \{g_i\}_{i \in I}$ be a frame sequence in $\mathcal H$ with bounds $A_2, B_2$. Suppose $0<(\sum \limits_{i \in I} \|f_i\|^2)^{\frac {1} {2}}=\lambda _1 <1~\text{and}~0<(\sum \limits_{i \in I} \|g_i\|^2)^{\frac {1} {2}}=\lambda _2 <1$ so that  $(\lambda _1 \sqrt B_1 + \lambda _2 \sqrt B_2) <A_1$. Then $\mathcal F ~\text{and}~\mathcal G$ are woven.

\end{thm}

\begin{proof} The proof will be completed with the following steps.\\
\noindent {\bf Step 1:} For every $\sigma \subset I$ and every $f \in \mathcal H$,
$$\|  \sum \limits_{i \in \sigma^c }  \langle f,f_i \rangle f_i -  \sum \limits_{i \in \sigma^c } \langle f,g_i \rangle g_i \| \leq  (\lambda _1 \sqrt B_1 + \lambda _2 \sqrt B_2) \|f\|.$$
\noindent \underline{proof of Step 1:} Using Proposition \ref{frame sequence}, for all $f \in \mathcal H$ and $\sigma \subset I$ we have,
\begin{eqnarray*}
 \|  \sum \limits_{i \in \sigma^c }  \langle f,f_i \rangle f_i -  \sum \limits_{i \in \sigma^c } \langle f,g_i \rangle g_i \|
& \leq & \|\sum \limits_{i \in \sigma^c }  \langle f,f_i \rangle f_i\| + \|\sum \limits_{i \in \sigma^c } \langle f,g_i \rangle g_i \|
\\ & \leq & \sum \limits_{i \in \sigma^c } \|\langle f,f_i \rangle f_i\| + \sum \limits_{i \in \sigma^c } \| \langle f,g_i \rangle g_i \|
\\ & \leq & \sum \limits_{i \in I } \|\langle f,f_i \rangle f_i\| + \sum \limits_{i \in I } \| \langle f,g_i \rangle g_i \|
\\ & \leq & (\sum_{i \in I } |\langle f,f_i \rangle|^2)^{\frac {1} {2}} (\sum_{i \in I} \|f_i\|^2)^{\frac {1} {2}}
\\ & + &
 (\sum_{i \in I} |\langle f,g_i \rangle|^2)^{\frac {1} {2}} (\sum_{i \in I } \|g_i\|^2)^{\frac {1} {2}}
\\ & \leq & (\lambda _1 \sqrt B_1 + \lambda _2 \sqrt B_2) \|f\|.
\end{eqnarray*}

\noindent {\bf Step 2:}  For every weaving the lower frame bound is $\frac {[A_1 - (\lambda _1 \sqrt B_1 + \lambda _2 \sqrt B_2)]^2} {B_1 + B_2}.$

\noindent \underline{proof of Step 2:} Applying Step 1, for every $f \in \mathcal H$ we obtain,
\begin{eqnarray*}
\|\sum \limits_{i \in \sigma } \langle f,f_i \rangle f_i + \sum_{i \in \sigma^c } \langle f,g_i \rangle g_i\| & \geq & \|\sum_{i \in I } \langle f,f_i \rangle f_i\| - \|\sum_{i \in \sigma ^c } \langle f,f_i \rangle f_i - \sum_{i \in \sigma^c } \langle f,g_i \rangle g_i\|
\\ & \geq & [A_1 - (\lambda _1 \sqrt B_1 + \lambda _2 \sqrt B_2] \|f\|.
\end{eqnarray*}

\noindent Therefore applying Lemma \ref{woven-equivalency}, our goal is executed.

\noindent Furthermore, universal upper frame bound of the weaving will be accomplished by utilizing the Proposition \ref{Bessel}.
\end{proof}

%\noindent \underline{\bf Problem:} Can we produce Proposition 3.9 and Theorem 3.13 in infinite dimensional case?

\section*{Acknowledgments}
\noindent The first author is highly indebted to the fiscal support of MHRD, Government of India. He also extends his massive gratitude to Dr. Manideepa Saha (NIT Meghalaya, India) and  Professor Kallol Paul (Jadavpur University, India) for their useful suggestions and comments to improve this article. The second author is supported by DST-SERB project MTR/2017/000797.

%\section*{References}

\bibliography{references-frame}

\begin{thebibliography}{10}
\expandafter\ifx\csname url\endcsname\relax
  \def\url#1{\texttt{#1}}\fi
\expandafter\ifx\csname urlprefix\endcsname\relax\def\urlprefix{URL }\fi
\expandafter\ifx\csname href\endcsname\relax
  \def\href#1#2{#2} \def\path#1{#1}\fi

\bibitem{Ga46}
D.~Gabor, Theory of communication, J.I.E.E. 93 (1946) 429--459.

\bibitem{DaGrMa86}
I.~Daubechies, A.~Grossmann, Y.~Mayer, Painless nonorthogonal expansions,
  Journal of Mathematical Physics 27~(5) (1986) 1271--1283.

\bibitem{Su06}
W.~Sun, {$G$}-frames and {$G$}-riesz bases, J. Math. Anal. Appl. 322~(1) (2006)
  437--452.

\bibitem{Ga12}
L.~G\v{a}vru\c{t}a, Frames for operators, Applied and Computational Harmonic
  Analysis 32~(1) (2012) 139--144.

\bibitem{CaKu04}
P.~Casazza, G.~Kutyniok, Frames of subspaces, Contemporary Math, AMS 345 (2004)
  87--114.

\bibitem{KhAs07}
A.~Khosravi, M.~S. Asgari, Frames of subspaces and approximation of the inverse
  frame operator, Houston Journal of Mathematics 33~(3) (2007) 907--920.

\bibitem{Bh17}
A.~Bhandari, S.~Mukherjee, Atomic subspaces for operators, Submitted,
  arXiv:1705.06042.

\bibitem{EaPa12}
T.~C. Easwaran~Nambudiri, K.~Parthasarathy, Generalised {W}eyl–{H}eisenberg
  frame operators, Bulletin des Sciences Math\'{e}matiques 136~(1) (2012)
  44--53.

\bibitem{EaPa13}
T.~C. Easwaran~Nambudiri, K.~Parthasarathy, A characterisation of
  {W}eyl–{H}eisenberg frame operators, Bulletin des Sciences
  Math\'{e}matiques 137~(3) (2013) 322--324.

\bibitem{MiCh18}
S.~Mishra, S.~Chakraborty, Dynamical system analysis of quintom dark energy
  model, The European Physical Journal C 78:917.

\bibitem{ArRa19}
S.~Arati, R.~Radha, Frames and riesz bases for shift invariant spaces on the
  abstract {H}eisenberg group, Indagationes Mathematicae 30~(1) (2019)
  106--127.

\bibitem{DeVa17}
Deepshikha, L.~Vashisht, Necessary and sufficient conditions for discrete
  wavelet frames in $\mathbb{C}^\mathbb{N}$, Journal of Geometry and Physics
  117 (2017) 134 -- 143.

\bibitem{EaPa18}
T.~C. Easwaran~Nambudiri, K.~Parthasarathy, Bessel sequences, wavelet frames,
  duals and extensions, Indagationes Mathematicae 29~(3) (2018) 907--915.

\bibitem{AbGi14}
L.~D. Abreu, J.~E. Gilbert, Wavelet-type frames for an interval, Expositiones
  Mathematicae 32~(3) (2014) 274--283.

\bibitem{GaSh19}
A.~Gamber, N.~K. Shukla, Pairwise orthogonal frames generated by regular
  representations of lca groups, Bulletin des Sciences Math\'{e}matiques 152
  (2019) 40--60.

\bibitem{BeCaGrLaLy16}
T.~Bemrose, P.~Casazza, K.~Gr{\"o}chenig, M.~Lammers, R.~Lynch, Weaving frames,
  Operators and Matrices 10~(4) (2016) 1093--1116.

\bibitem{DeVaVe17}
Deepshikha, L.~Vashisht, G.~Verma, Generalized weaving frames for operators in
  {H}ilbert spaces, Results in Mathematics 72~(3) (2017) 1369 -- 1391.

\bibitem{VaDe16}
L.~Vashisht, Deepshikha, Weaving properties of generalized continuous frames
  generated by an iterated function system, Journal of Geometry and Physics 110
  (2016) 282 -- 295.

\bibitem{DeVa18}
Deepshikha, L.~Vashisht, Vector-valued (super) weaving frames, Journal of
  Geometry and Physics 134 (2018) 48 -- 57.

\bibitem{CaKu12}
P.~Casazza, G.~Kutyniok, Finite Frames: Theory and Applications, Applied and
  Numerical Harmonic Analysis, Birkh{\"a}user Boston, 2012.

\bibitem{Ch08}
O.~Christensen, Frames and Bases-An Introductory Course, Birkh{\"a}user,
  Boston, 2008.

\bibitem{HaKoLaWe07}
D.~Han, K.~Kornelson, D.~R. Larson, E.~Weber, Frames for Undergraduates, AMS,
  2007.

\bibitem{BaTr09}
O.~M. Baksalary, G.~Trenkler, On angles and distances between subspaces, Linear
  Algebra and its Applications 431 (2009) 2243--2260.

\bibitem{De95}
F.~Deutsch, The angle between subspaces of a {H}ilbert space, Approximation
  Theory, Wavelets and Applications 454 (1995) 107 -- 130.

\bibitem{Ka80}
T.~Kato, Perturbation Theory for Linear Operators, Springer, New York, 1980.

\bibitem{Do66}
R.~G. Douglas, On majorization, factorization and range inclusion of operators
  on {H}ilbert space, Proc. Amer Math. Society 17~(2) (1966) 413--415.

\end{thebibliography}

\end{document}